\documentclass[a4, 12pt]{amsart}
\usepackage{amssymb}
\usepackage{amstext}
\usepackage{amsmath}
\usepackage{amscd}
\usepackage{latexsym}
\usepackage{amsfonts}

\theoremstyle{plain}
\newtheorem{thm}{Theorem}[section]
\newtheorem*{thm*}{Theorem}
\newtheorem*{cor*}{Corollary}

\newtheorem{prop}[thm]{Proposition}
\newtheorem{lem}[thm]{Lemma}
\newtheorem{cor}[thm]{Corollary}
\newtheorem{claim}{Claim}
\newtheorem*{claim*}{Claim}

\theoremstyle{definition}
\newtheorem{defn}[thm]{Definition}

\newtheorem{rem}[thm]{Remark}

\theoremstyle{remark}

\numberwithin{equation}{thm}

\def\Coker{\mathrm{Coker}}

\def\rank{\mathrm{rank}}
\def\a{\mathfrak a}
\def\b{\mathfrak b}
\def\e{\mathrm{e}}
\def\m{\mathfrak m}

\def\p{\mathfrak p}

\def\Z{\Bbb Z}

\def\H{\mathrm{H}}

\newcommand{\rmC}{\mathrm{C}}

\newcommand{\rmG}{\mathrm{G}}

\newcommand{\rmR}{\mathrm{R}}
\newcommand{\rmS}{\mathrm{S}}

\newcommand{\calF}{\mathcal{F}}

\newcommand{\calK}{\mathcal{K}}

\newcommand{\fkm}{\mathfrak{m}}

\def\depth{\mathrm{depth}}

\def\Ann{\mathrm{Ann}}
\def\Ass{\mathrm{Ass}}

\def\height{\mathrm{ht}}

\newcommand{\mapright}[1]{%
\smash{\mathop{%
\hbox to 1cm{\rightarrowfill}}\limits^{#1}}}

\newcommand{\mapleft}[1]{%
\smash{\mathop{%
\hbox to 1cm{\leftarrowfill}}\limits_{#1}}}

\begin{document}

\setlength{\baselineskip}{12.7pt}
\title{The structure of the Sally module of\\ integrally closed ideals}
\pagestyle{plain}
\author{Kazuho Ozeki}
\address{Department of Mathematical Science, Faculty of Science, Yamaguchi University, 1677-1 Yoshida, Yamaguchi 753-8512, Japan}
\email{ozeki@yamaguchi-u.ac.jp}
\author{Maria Evelina Rossi}
\address{Dipartimento di Matematica, Universita di Genova, Via Dodecaneso, 35-16146, Genova, Italy}
\email{rossim@dima.unige.it}
\dedicatory{To Shiro Goto on the occasion of his seventieth birthday}
\thanks{
\endgraf
}
\maketitle
\begin{abstract}
The first two Hilbert coefficients of a primary ideal play an important role in commutative algebra and in algebraic geometry. In this paper we give a complete algebraic structure   of the  Sally module  of  integrally closed ideals  $I$ in a Cohen-Macaulay local ring $A$ satisfying the equality $\e_1(I)=\e_0(I)-\ell_A(A/I)+\ell_A(I^2/QI)+1, $ where $Q$ is a minimal reduction of $I$, and $\e_0(I)$ and $\e_1(I)$ denote the first two Hilbert coefficients of $I, $ respectively the multiplicity and the Chern number of $I.$ This almost extremal value of $\e_1(I) $ with respect classical inequalities holds a complete description of the homological and the numerical invariants  of the associated graded ring. Examples are given.
\end{abstract}


\section{Introduction and notation }

Throughout this paper, let $A$ denote a Cohen-Macaulay local ring with maximal ideal $\m$ and positive dimension $d.$ 
 Let $I$ be an $\m$-primary ideal in $A$  and,  for simplicity,  we assume the residue class field $A/\m$ is infinite.
Let $\ell_A(N)$ denote, for an $A$-module $N$, the length of $N$.
The  integers $\{\e_i(I)\}_{0 \leq i \leq d}$ such that the equality
$$\ell_A(A/I^{n+1})={\e}_0(I)\binom{n+d}{d}-{\e_1}(I)\binom{n+d-1}{d-1}+\cdots+(-1)^d{\e}_d(I)$$
holds true for all integers $n \gg 0$, are called  the {\it Hilbert coefficients} of $A$ with respect to $I$.  This polynomial, known as the Hilbert-Samuel  polynomial of $I $  and denoted by $HP_I(n), $  encodes the asymptotic information coming from the Hilbert function  $H_I(t)$
of $I$ which  is defined as
$$H_I(t) =\ell_A(I^t/I^{t+1}). $$
The generating function
of the numerical function $H_I(t) $ is the power series $$\ \ HS_I(z)=\sum_{t\ge
0}H_I(t)z^t.$$ This series is
called the Hilbert   series of $I.$ It is well known that this
series is rational and that, even
more, there exists a polynomial $h_I(z)$ with integers coefficients such that
$h_I(1)\not= 0$ and
$$  HS_I(z)=\frac {h_I(z)}{(1-z)^d}.$$
 Notice that for all  $i\ge 0, $ the Hilbert coefficients can be computed as it follows $$\e_i(I):=\frac{h_I^{(i)}(1)}{i!}  $$
 where $h_I^{(i)}$ denotes the $i-$th derivative of $h_I(z) $ evaluated at $i$  ($h^{(0)}=h_I$).

\noindent Choose a parameter ideal $Q $ of $A$ which forms a reduction of $I$ and 
let
$$ R=\rmR(I) := A[I t] \ \ ~~\operatorname{and}~~ \ T= \rmR(Q):= A[Qt] \ \ \subseteq A[t]$$
denote, respectively, the Rees algebras of $I$ and $Q$. Let
$$ R'=\rmR'(I) := A[I t,t^{-1}] \ \subseteq A[t,t^{-1}] \ \ \ ~~\operatorname{and}~~ \ G=\rmG(I):= R'/t^{-1}R' \cong \bigoplus_{n \geq 0}I^n/I^{n+1}.$$
Following Vasconcelos \cite{V}, we consider 
$$S=\rmS_Q(I)=IR/IT \cong \bigoplus_{n \geq 1}I^{n+1}/Q^nI$$ 
the Sally module of $I$ with respect to $Q$.

\vskip 2mm
The notion of  {\it{ filtration of the Sally module}}  was introduced by M. Vaz Pinto \cite{VP} as follows.
We denote by $E(\alpha)$, for a graded $T$-module $E$ and each $\alpha \in \Z$, the graded $T$-module whose grading is given by $[E(\alpha)]_n=E_{\alpha+n}$ for all $n \in \Z$.

\begin{defn}\label{VP}(\cite{VP})
We set, for each $i \geq 1$,
$$ C^{(i)}=(I^iR/I^iT)(-i+1) \cong \bigoplus_{n \geq i} I^{n+1}/Q^{n-i+1}I^{i}. $$
and let $L^{(i)}=  T[C^{(i)}]_i$.
Then, because $L^{(i)} \cong \bigoplus_{n \geq i} Q^{n-i}I^{i+1}/Q^{n-i+1}I^i$ and $C^{(i)}/L^{(i)} \cong C^{(i+1)}$ as graded $T$-modules, we have the following natural exact sequences of graded $T$-modules
$$ 0 \to L^{(i)} \to C^{(i)} \to C^{(i+1)} \to 0 $$
for every $i \ge 1.$ 
\end{defn}
We notice that $C^{(1)}=S$, and $C^{(i)}$ are finitely generated graded $T$-modules for all $i \geq 1$, since $R$ is a module-finite extension of the graded ring $T$.
 
So,  from now on,   we set
$$C = \rmC_Q(I)=C^{(2)}= (I^2R/I^2T)(-1) $$
and  we shall explore the structure of   $C. $  
Assume  that $I$ is integrally closed. 
Then, by \cite{EV, GR}, the inequality
$$ \e_1(I) \geq \e_0(I) -\ell_A(A/I) + \ell_A(I^2/QI)$$
holds true and the equality $ \e_1(I) = \e_0(I)-\ell_A(A/I)+ \ell_A(I^2/QI)$ holds  if and only if $I^3=QI^2$.
When this is the case, the associated graded ring $G$ of $I$ is Cohen-Macaulay and the behavior of the Hilbert-Samuel  function $\ell_A(A/I^{n+1})$ of $I$ is known (see \cite{EV}, Corollary \ref{EV}).
Thus the integrally closed ideal $I$ with $ \e_1(I) = \e_0(I)-\ell_A(A/I)+ \ell_A(I^2/QI)$ enjoys nice properties and it seems natural to ask what happens on the integrally closed ideal $I$ which satisfies  the equality $ \e_1(I) = \e_0(I)-\ell_A(A/I)+ \ell_A(I^2/QI)+1$.
The problem is not trivial even if we consider $d=1.$ 

\vskip 2mm
We notice here that $\ell_A(I^2/QI)=\e_0(I)+(d-1)\ell_A(A/I)-\ell_A(I/I^2)$ holds true (see for instance \cite{RV3}), so that  $\ell_A(I^2/QI)$ does not depend on a minimal reduction $Q$ of $I$.
 
\vskip 2mm
 Let $B=T/\m T \cong (A/\m)[X_1,X_2,\cdots,X_d]$ which is a  polynomial ring with $d$ indeterminates over the field $A/\m$. The main result of this paper is stated as follows.

\begin{thm}\label{main1}
Assume that $I$ is integrally closed.
Then the following  conditions are equivalent: 
\begin{itemize}
\item[(1)] $\e_1(I)=\e_0(I)-\ell_A(A/I)+\ell_A(I^2/QI)+1$,
\item[(2)] $\m C=(0)$ and $\rank_B C=1$,  
\item[(3)] $C \cong (X_1,X_2,\cdots,X_c)B(-1)$ as graded $T$-modules for some $1 \leq c \leq d$, where $X_1,X_2,\cdots,X_d$ are linearly independent linear forms of the polynomial ring $B$.
\end{itemize}
When this is the case, $c=\ell_A(I^3/QI^2)$ and $I^4=QI^3$, and the following assertions hold true: 
\begin{itemize}
\item[(i)] $\depth ~G \geq d-c$ and $\depth_TC=d-c+1$,
\item[(ii)] $\depth~ G=d-c$, if $c \geq 2$.
\item[(iii)] Suppose $c=1 <d$. Then  $HP_I(n) = \ell_A(A/I^{n+1}) $ for all $n \geq 0$ and 

\[ \e_i(I) =  \left\{
\begin{array}{ll}
\e_1(I)-\e_0(I)+\ell_A(A/I)+1 & \quad \mbox{if $i=2 $,} \\
1 & \quad \mbox{if $i =3$ and  $d \ge 3$,} \\
0 & \quad \mbox{if $4 \le i \le d.$}
\end{array}
\right.\]
 
\item[(iv)] Suppose $2 \leq c < d$. Then  $HP_I(n) = \ell_A(A/I^{n+1}) $ for all $n \geq 0$ and 

\[ \e_i(I) =  \left\{
\begin{array}{ll}
\e_1(I)-\e_0(I)+\ell_A(A/I) & \quad \mbox{if $i=2 $,} \\
0 & \quad \mbox{if $i \neq c+1, c+2 $,\ \ $3 \leq i \leq d$} \\
(-1)^{c+1} & \quad \mbox{if $i =c+1, c+2$,\ \  $3 \leq i \leq d$}
\end{array}
\right.\]

 \item[(v)] Suppose $c=d$. Then  $HP_I(n) = \ell_A(A/I^{n+1}) $ for all $n \geq 2 $ and 

\[ \e_i(I) =  \left\{
\begin{array}{ll}
\e_1(I)-\e_0(I)+\ell_A(A/I) & \quad \mbox{if $i=2$ and $d \geq 2$,} \\
0 & \quad \mbox{if  $3 \leq i \leq d$}  
\end{array}
\right.\]
 
 \item[(vi)] The Hilbert series $ HS_I(z)$   is given by {\small { $$ HS_I(z)=\frac{\ell_A(A/I)+\{\e_0(I)-\ell_A(A/I)-\ell_A(I^2/QI)-1\}z+\{\ell_A(I^2/QI)+1\}z^2+(1-z)^{c+1}z}{(1-z)^d}.$$}} 
\end{itemize}
\end{thm}

\vskip 3mm

Let us briefly explain how this paper is organized.
We shall prove Theorem \ref{main1} in Section 3.
In Section 2 we will introduce some auxiliary results on the structure of the $T$-module  $C=\rmC_Q(I)=(I^2R/I^2T)(-1)$, some of them are stated in a general setting.    Our hope is that these information will be successfully applied to give new insights in problems related to the structure of Sally's module.    In Section 4 we will introduce some consequences of Theorem \ref{main1}. In particular 
we shall explore the integrally closed ideals  $I$ with $\e_1(I) \le \e_0(I)-\ell_A(A/I)+3. $     
 In Section 5 we will construct a class of Cohen-Macaulay  local rings satisfying  condition (1) in Theorem \ref{main1}.

\section{Preliminary Steps}

The purpose of this section is to summarize some  results on the structure of the graded $T$-module  $C=\rmC_Q(I)=(I^2R/I^2T)(-1)$, which we need throughout this paper. Remark that in this section $I$ is an $\m$-primary ideal not necessarily integrally closed.  


\begin{lem}\label{fact1}
The following assertions hold true.
\begin{itemize}
\item[$(1)$] $\m^{\ell} C = (0)$ for integers $\ell \gg 0$; hence ${\dim}_TC \leq d$.
\item[$(2)$]The homogeneous components $\{ C_n \}_{n \in \Z}$ of the graded $T$-module $C$ are given by
\[ C_n \cong  \left\{
\begin{array}{rl}
(0) & \quad \mbox{if $n \leq 1 $,} \\
I^{n+1}/Q^{n-1}I^2 & \quad \mbox{if $n \geq 2$.}
\end{array}
\right.\]
\item[$(3)$] $C=(0)$ if and only if $I^3=QI^2$.
\item[$(4)$] $\m C=(0)$ if and only if $\m I^{n+1} \subseteq Q^{n-1}I^2$ for all $n \geq 2$.
\item[$(5)$] $S = TC_2$ if and only if $I^4 = QI^3$. 
\end{itemize}
\end{lem}

\begin{proof}
(1) Let $u = t^{-1}$ and $T'=\rmR'(Q)$.
Notice that $C = (I^2R/I^2T)(-1) \cong (I^2R'/I^2T')(-1)$ as graded $T$-modules. 
We then have $u^{\ell}{\cdot}(I^2R'/I^2T') = (0)$ for some $\ell \gg 0$, because the graded $T'$-module $I^2R'/I^2T'$ is finitely generated and $[I^2R'/I^2T']_n =(0)$ for all $n \leq 0$.
Therefore, $\m^{\ell}C = (0)$ for $\ell \gg 0$, because $Q^{\ell} = (Qt^{\ell})u^{\ell} \subseteq u^{\ell}T' \cap A$ and $\m = \sqrt{Q}$.

(2) Since $[I^2R]_n = (I^{n+2})t^n$ and $[I^2T]_n = (I^2Q^n)t^n$ for all $n \geq 0$, assertion (2) follows from the definition of the module $C = (I^2R/I^2T)(-1)$. 

Assertions (3), (4), and (5) are readily follow from assertion (2). 
\end{proof}

In the following result  we need that $Q \cap I^2=QI$ holds true.
This condition is automatically satisfied in the case where $I$ is integrally closed (see \cite{H, I2}).

\begin{prop}\label{Ass}
Suppose that $Q \cap I^2=QI$.
Then we have $\Ass_T C  \subseteq \{\m T \}$ so that $\dim_TC=d$, if $C \neq (0)$.
\end{prop}

Let $Q=(a_1, \dots, a_d) $ be a minimal reduction of $I. $  In the proof of Proposition \ref{Ass} we need the following lemmata.

\begin{lem}\label{cap}
Suppose  that $Q \cap I^2=QI$.
Then we have $(a_1,a_2,\cdots,a_i) \cap Q^{n+1}I^2=(a_1,a_2,\cdots,a_i)Q^nI^2$ for all $n \geq 0$ and $1 \leq i \leq d$.
Therefore, $a_1t \in T$ is a nonzero divisor on $T/I^2T$.
\end{lem}

\begin{proof}
We have only to show that   $(a_1,a_2,\cdots,a_i) \cap Q^{n+1}I^2 \subseteq (a_1,a_2,\cdots,a_i)Q^n I^2$ holds true for all $n \geq 0$.
We proceed by induction on $n$ and $i$.
We may assume that $i < d$ and that our assertion holds true for $i+1$.
Suppose $n=0$ and take $x \in (a_1,a_2,\cdots,a_i) \cap QI^2$. 
Then, by the hypothesis of induction on $i$, we have
\begin{eqnarray*}
(a_1,a_2,\cdots,a_i) \cap QI^2 &\subseteq& (a_1,a_2,\cdots,a_i,a_{i+1}) \cap QI^2
= (a_1,a_2,\cdots,a_i,a_{i+1}) I^2
\end{eqnarray*}
Then we write $x=y+a_{i+1}z$ with $y \in (a_1,a_2,\cdots,a_i) I^2$ and $z \in I^2$.
Since $z \in [(a_1,a_2,\cdots,a_i):a_{i+1}]\cap I^2=(a_1,a_2,\cdots,a_i) \cap I^2 \subseteq (a_1,a_2,\cdots,a_i) \cap QI$.
Notice that, since $Q/QI \cong (A/I)^d$, $(a_1,a_2,\cdots,a_i) \cap QI=(a_1,a_2,\cdots,a_i)I$ holds true.
Hence, we get $x =y+a_{i+1}z \in (a_1,a_2,\cdots,a_i)I^2$.
Therefore we have $(a_1,a_2,\cdots,a_i) \cap QI^2=(a_1,a_2,\cdots,a_i)I^2. $  

Assume that $n \geq 1$ and that our assertion holds true for $n-1$.
Take $x \in (a_1,a_2,\cdots,a_i) \cap Q^{n+1}I^2$.
Then, by the hypothesis of induction on $i$, we have
\begin{eqnarray*}
(a_1,a_2,\cdots,a_i) \cap Q^{n+1}I^2 &\subseteq& (a_1,a_2,\cdots,a_i,a_{i+1}) \cap Q^{n+1}I^2
= (a_1,a_2,\cdots,a_i,a_{i+1})Q^n I^2
\end{eqnarray*}
Write $x=y+a_{i+1}z$ with $y \in (a_1,a_2,\cdots,a_i)Q^n I^2$ and $z \in Q^n I^2$.
By the hypothesis of induction on $n$, we have $z \in [(a_1,a_2,\cdots,a_i):a_{i+1}]\cap Q^{n} I^2=(a_1,a_2,\cdots,a_i) \cap Q^{n} I^2=(a_1,a_2,\cdots,a_i) Q^{n-1} I^2$.
Therefore, we get $x =y+a_{i+1}z \in (a_1,a_2,\cdots,a_i) Q^n I^2$. 
Thus, $(a_1,a_2,\cdots,a_i) \cap Q^{n+1}I^2=(a_1,a_2,\cdots,a_i)Q^n I^2$ as required.
\end{proof}

\begin{lem}\label{cap2}
Suppose that $Q \cap I^2=QI$.
Then $Q^{n+1} \cap Q^nI^2=Q^{n+1}I$ for all $n \geq 0$.
\end{lem}

\begin{proof}
We have only to show that $Q^{n+1} \cap Q^nI^2 \subseteq Q^{n+1}I$ for $n \geq 0$.
Take $f \in Q^{n+1} \cap Q^nI^2$ and write $f=\sum_{|\alpha|=n}x_{\alpha}a_1^{\alpha_1} a_2^{\alpha_2} \cdots a_d^{\alpha_d}$ with $x_{\alpha} \in I^2$, where $\alpha=(\alpha_1,\alpha_2,\cdots,\alpha_d) \in \Z^d$ and $|\alpha|=\sum_{i=1}^d\alpha_i$.
Then, we have $\overline{ft^n}=\sum_{|\alpha|=n}\overline{x_{\alpha}a_1^{\alpha_1} a_2^{\alpha_2} \cdots a_d^{\alpha_d}t^n}=0$ in $\rmG(Q)=\rmR'(Q)/t^{-1}\rmR'(Q)$, where $\overline{ft^n}$ and $\overline{x_{\alpha}a_1^{\alpha_1} a_2^{\alpha_2} \cdots a_d^{\alpha_d}t^n}=0$ denote the images of $ft^n$ and $x_{\alpha}a_1^{\alpha_1} a_2^{\alpha_2} \cdots a_d^{\alpha_d}t^n$ in $\rmG(Q)$ respectively.
Because $\rmG(Q) \cong (A/Q)[\overline{a_1t}, \overline{a_2t}, \cdots, \overline{a_dt}]$ is the polynomial ring over the ring $A/Q$, we have $x_{\alpha} \in Q \cap I^2=QI$.
Thus $f \in Q^{n+1}I$ so that $Q^{n+1} \cap Q^nI^2 = Q^{n+1}I$ holds true.
\end{proof}

\begin{lem}\label{exact}
Suppose that $Q \cap I^2=QI$.
Then the following  sequences
\begin{itemize}
\item[(1)] $ 0 \to T/I T \overset{a_1}{\to} T/I^2 T \to T/[I^2T+a_1T] \to 0 $, and
\item[(2)] $ 0 \to T/[I^2 T+QT](-1) \overset{a_1t}{\to} T/[I^2T+a_1T] \to T/[I^2 T+a_1t T+a_1T] \to 0 $
\end{itemize}
of graded $T$-modules are exact. 
\end{lem}

\begin{proof}
$(1)$ Let us consider the homomorphism $$\phi: T \to T/I^2 T$$ of graded $T$-modules such that $\phi(f)=\overline{a_1 f}$ for $f \in T$ where $\overline{a_1 f}$ denotes the image of $a_1 f$ in $T/I^2 T$.
Because $\phi(IT)=(0)$ and $\Coker \phi=T/[I^2 T +a_1T] $, we have only to show that $\ker \phi \subseteq IT$.
Take $f \in [\ker \phi]_{n}$ and write $f=xt^{n}$ with $n \geq 0$ and $x \in  Q^{n}$.
Then we have $a_1 f=a_1xt^n \in I^2 T $ so that $a_1x \in Q^n I^2 \cap Q^{n+1}=Q^{n+1}I$ by Lemma \ref{cap2}.
Because $a_1t \in T$ forms a nonzero divisor on $T/IT$ (notice that $T/IT \cong (A/I)[\overline{a_1t}, \overline{a_2t}, \cdots, \overline{a_dt}]$ is the polynomial ring over the ring $A/I$), we have $(a_1) \cap Q^{n+1}I=a_1Q^nI$ so that
$x \in Q^nI$.
Therefore $f \in IT$, and hence $\ker \phi \subseteq IT$.
Thus, we get the first required exact sequence.

$(2)$ Let us consider the homomorphism $$\varphi: T(-1) \to T/[I^2 T+a_1T]$$ of graded $T$-modules such that $\varphi(f)=\overline{a_1tf}$ for $f \in T$ where $\overline{a_1tf}$ denotes the image of $a_1t f$ in $T/[I^2 T+a_1T]$.
Because $\varphi(I^2T+QT)=(0)$ and $\Coker \phi=T/[I^2 T +a_1tT+a_1T] $, we need to show that $[\ker \varphi]_n \subseteq I^2T_{n-1}+QT_{n-1}$ for all $n \geq 1$.
Take $f \in [\ker \varphi]_{n}$ and write $f=xt^{n-1}$ with $n \geq 1$ and $x \in  Q^{n-1}$.
Then we have $a_1tf=a_1xt^n \in I^2 T+a_1 T $ so that $a_1x \in Q^n I^2+a_1Q^n$.
Write $a_1x=y+a_1z$ with $y \in Q^nI^2$ and $z \in Q^n$.
Then have $$a_1(x-z)=y \in (a_1) \cap Q^n I^2=a_1Q^{n-1}I^2$$ by Lemma \ref{cap}.
Hence $x-z \in Q^{n-1}I^2$ so that $x \in Q^{n-1}I^2+Q^n$.
Therefore $f \in I^2T_{n-1}+QT_{n-1}$ and hence $[\ker \varphi]_n \subseteq I^2T_{n-1}+QT_{n-1}$.
Consequently, we get the second  required exact sequence.
\end{proof}

The following Lemma \ref{Ass_} is the crucial fact in the proof of Proposition \ref{Ass}.

\begin{lem}\label{Ass_}
Assume that $Q \cap I^2=QI$.
Then we have $\Ass_T (T/I^2T) = \{ \m T \}$.
\end{lem}

\begin{proof}
Take $P \in \Ass_T ( T/I^2T)$ then we have $\m T \subseteq P$, because $\m^{\ell} (T/I^2T)=(0)$ for $\ell \gg 0$.
Assume that $\m T \subsetneq P$ then $\height_TP \geq 2$, because $\m T$ is a height one prime ideal in $T$. 
Consider  the following exact sequences
$$ 0 \to T_P/I T_P \to  T_P/I^2 T_P \to T_P/[I^2T_P+a_1T_P] \to 0 \ \ \ (*_1) \ \ \ \mbox{and} $$
$$ 0 \to T_P/[I^2 T_P+QT_P] \to T_P/[I^2T_P+a_1T_P] \to T_P/[I^2 T_P+a_1t T_P+a_1T_P] \to 0 \ \ \ (*_2) $$
of $T_P$-modules, which follow from the exact sequences in Lemma \ref{exact}.
Then, since $\depth_{T_P} T_P/I T_P > 0$ $($notice that $T/IT \cong (A/I)[X_1,X_2,\cdots,X_d]$ is the polynomial ring with $d$ indeterminates over the ring $A/I$$)$ and $\depth_{T_P} T_P/I^2T_P=0$, we have $\depth_{T_P}T_P/[I^2T_P+a_1T_P]=0$ by the exact sequence $(*_1)$.
Notice that $T/[I^2T+QT] \cong (A/[I^2+Q])[X_1,X_2,\cdots,X_d]$ is the polynomial rings with $d$ indeterminates over the ring $A/[I^2+Q]$.
Hence, $\depth_{T_P} T_P/[I^2 T_P+a_1t T_P+a_1T_P]=0$ by the exact sequence $(*_2)$.
Then, because $a_1t \in \Ann_T(T/[I^2 T+a_1t T+a_1T])$, we have $a_1t \in P$.
Therefore $\depth_{T_P} T_P/I^2 T_P>0$ by Lemma \ref{cap}, however it is impossible. 
Thus, $P=\m T$ as required.
\end{proof}

Let us now give a proof of Proposition \ref{Ass}.

\begin{proof}[Proof of Proposition \ref{Ass}]
Take $P \in \Ass_TC$. Then we have $\m T \subseteq P$, because $\m^{\ell} C=(0)$ for some $\ell \gg 0$ by Lemma \ref{fact1} (1).
Suppose that $\m T \subsetneq P$ then $\height_TP \geq 2$, because $\m T$ is a height one prime ideal in $T$.
We look at the following exact sequences
$$ 0 \to I^2 T_P \to I^2 R_P \to C_P \to 0 \ \ \ (*_3) \ \ \ 
\mbox{and} \ \ \ 0 \to I^2 T_P \to T_P \to  T_P/I^2 T_P \to 0 \ \ \ (*_4) $$
of $T_P$-modules which follows from the canonical exact sequences
$$ 0 \to I^2 T(-1) \to I^2 R(-1) \to C \to 0 \ \ \ \mbox{and} \ \ \ 0 \to I^2 T \to  T \to T/I^2 T \to 0 $$
of $T$-modules.
We notice here that $\depth_{T_P} I^2 R_P > 0 $, because $a_1 \in A$ is a nonzero divisor on $I^2 R$.
Thanks to the depth lemma and the exact sequence $(*_3)$, we get $\depth_{T_P}I^2 T_P=1$, because $\depth_{T_P}I^2 R_P>0$ and $\depth_{T_P}C_P=0$.
Then, since $\depth_{T_P} T_P \geq 2$, we have $\depth_{T_P}T_P/I^2 T_P=0$ by the exact sequence $(*_4)$.
Therefore, we have $P=\m T$ by Lemma \ref{Ass_}, which is impossible. 
Thus, $P=\m T$ as required.
\end{proof}

The following techniques are due to M. Vaz Pinto \cite[Section 2]{VP}.

Let $L=L^{(1)}=TS_1$ then $L \cong \bigoplus_{n \geq 1}Q^{n-1}I^{2}/Q^{n}I$ and $S/L \cong C$ as graded $T$-modules.
Then there exist a canonical exact sequence
$$ 0 \to L \to S \to C \to 0 \ \ \ (\dagger) $$
of graded $T$-modules $($Definition \ref{VP}$)$.
We set $D=(I^2/QI) \otimes_A (T/\Ann_A(I^2/QI)T)$.
Notice here that $D$ forms a graded $T$-module and $T/\Ann_A(I^2/QI)T \cong (A/\Ann_A(I^2/QI))[X_1,X_2,\cdots,X_d]$ is the polynomial ring with $d$ indeterminates over the ring $A/\Ann_A(I^2/QI)$.
Let
$$ \theta:D(-1) \to L$$
denotes an epimorphism of graded $T$-modules such that $\theta(\sum_{\alpha} \overline{x_{\alpha}} \otimes X_1^{\alpha_1}X_2^{\alpha_2} \cdots X_d^{\alpha_d})=\sum_{\alpha}\overline{x_{\alpha}a_1^{\alpha_1}a_2^{\alpha_2}\cdots a_d^{\alpha_d}t^{|\alpha|+1}}$ for $x_{\alpha} \in I^2$ and $\alpha=(\alpha_1,\alpha_2,\cdots,\alpha_d) \in \Z^d$ with $\alpha_i \geq 0$ $(1 \leq i \leq d)$, where $|\alpha|=\sum_{i=1}^d\alpha_i$, and $\overline{x_{\alpha}}$ and $\overline{x_{\alpha}a_1^{\alpha_1}a_2^{\alpha_2}\cdots a_d^{\alpha_d}t^{|\alpha|+1}}$ denote the images of $x_{\alpha}$ in $I^2/QI$ and $x_{\alpha}a_1^{\alpha_1}a_2^{\alpha_2}\cdots a_d^{\alpha_d}t^{|\alpha|+1}$ in $L$.

Then we have the following lemma.

\begin{lem}\label{L1}
Suppose that $Q \cap I^2=QI$.
Then the map $\theta:D(-1) \to L$ is an isomorphism of graded $T$-modules.
\end{lem}

\begin{proof}
We have only to show that $\ker \theta=(0)$.
Assume that $\ker \theta \neq (0)$ and let $n \geq 2$ as the least integer so that $[\ker \theta]_n \neq (0)$ (notice that $[\ker \theta]_n=(0)$ for all $n \leq 1$).
Take $0 \neq g \in [\ker \theta]_n$ and we set
$$\Gamma=\{(\alpha_1,\alpha_2,\cdots,\alpha_{d-1}, 0) \in \Z^d \ | \ \mbox{$\alpha_i \geq 0$ for $1 \leq i \leq d-1$ and $\sum_{i=1}^{d-1}\alpha_i=n-1$} \} \ \ \ \mbox{and}$$
$$\Gamma'=\{(\beta_1,\beta_2,\cdots,\beta_d) \in \Z^d \ | \ \mbox{$\beta_i \geq 0$ for $1 \leq i \leq d-1$, $\beta_d \geq 1$, and $\sum_{i=1}^{d} \beta_i=n-1$} \}.$$
Then because
$$\Gamma \cup \Gamma'=\{(\alpha_1, \alpha_2,\cdots, \alpha_d) \in \Z^d \ | \ \mbox{$\alpha_i \geq 0$ for $1 \leq i \leq d$ and $\sum_{i=1}^{d}\alpha_i=n-1$} \}$$
we may write 
\begin{eqnarray*}
g &=& \sum_{\alpha \in \Gamma \cup \Gamma'}\overline{x_{\alpha}} \otimes X_1^{\alpha_1}X_2^{\alpha_2} \cdots X_d^{\alpha_d}\\
&=& \sum_{\alpha \in \Gamma} \overline{x_{\alpha}} \otimes X_1^{\alpha_1}X_2^{\alpha_2} \cdots X_{d-1}^{\alpha_{d-1}}
+\sum_{\beta \in \Gamma'} \overline{x_{\beta}} \otimes X_1^{\beta_1}X_2^{\beta_2} \cdots X_d^{\beta_d}, 
\end{eqnarray*}
where $\overline{x_{\alpha}}, \overline{x_{\beta}}$ denote the images of $x_{\alpha}, x_{\beta} \in I^2$ in $I^2/QI$, respectively.
We may assume that $\sum_{\beta \in \Gamma'}\overline{x_{\beta}} \otimes X_1^{\beta_1}X_2^{\beta_2} \cdots X_d^{\beta_d} \neq 0$ in $D$.
Then we have
\begin{eqnarray*}
\theta(g)=\sum_{\alpha \in \Gamma}\overline{x_{\alpha}a_1^{\alpha_1}a_2^{\alpha_2} \cdots a_{d-1}^{\alpha_{d-1}}t^n}+
\overline{\sum_{\beta \in \Gamma'} x_{\beta}a_1^{\beta_1}a_2^{\beta_2} \cdots a_d^{\beta_d}t^n}=0
\end{eqnarray*}
so that
\begin{eqnarray*}
\sum_{\alpha \in \Gamma} x_{\alpha} a_1^{\alpha_1} a_2^{\alpha_2} \cdots a_{d-1}^{\alpha_{d-1}}+\sum_{\beta \in \Gamma'} x_{\beta}a_1^{\beta_1}a_2^{\beta_2} \cdots a_d^{\beta_d} \in Q^n I.
\end{eqnarray*}
Because $Q^n I=(a_1,a_2,\cdots,a_{d-1})^n I +a_dQ^{n-1} I$ and $\beta_d \geq 1$, we may write 
$$\sum_{\alpha \in \Gamma} x_{\alpha} a_1^{\alpha_1} a_2^{\alpha_2} \cdots a_{d-1}^{\alpha_{d-1}}+a_d\left(\sum_{\beta \in \Gamma'} x_{\beta}a_1^{\beta_1}a_2^{\beta_2} \cdots a_d^{\beta_d-1}\right)= \tau +a_d \rho$$
with $\tau \in (a_1,a_2,\cdots,a_{d-1})^n I$ and $\rho \in Q^{n-1} I$.
Then because $$ a_d\left(\sum_{\beta \in \Gamma'} x_{\beta}a_1^{\beta_1}a_2^{\beta_2} \cdots a_d^{\beta_d-1}-\rho\right)=\tau -\sum_{\alpha \in \Gamma} x_{\alpha} a_1^{\alpha_1} a_2^{\alpha_2} \cdots a_{d-1}^{\alpha_{d-1}} \in (a_1,a_2,\cdots,a_{d-1})^{n-1} $$
we have, by Lemma \ref{cap2},
$$\sum_{\beta \in \Gamma'} x_{\beta}a_1^{\beta_1}a_2^{\beta_2} \cdots a_d^{\beta_d-1}-\rho \in (a_1,a_2,\cdots,a_{d-1})^{n-1} \cap Q^{n-2}I^{2} \subseteq Q^{n-1} I.$$
Therefore, $\sum_{\beta \in \Gamma'} x_{\beta}a_1^{\beta_1}a_2^{\beta_2} \cdots a_d^{\beta_d-1} \in Q^{n-1}I$ and hence $\theta(\sum_{\beta \in \Gamma'}\overline{x_{\beta}} \otimes X_1^{\beta_1}X_2^{\beta_2} \cdots X_d^{\beta_d-1})=0$.
Then, because $\sum_{\beta \in \Gamma'} \overline{x_{\beta}} \otimes X_1^{\beta_1}X_2^{\beta_2} \cdots X_d^{\beta_d-1} \in [\ker \theta]_{n-1}=(0)$, we have $\sum_{\beta \in \Gamma'}\overline{x_{\beta}} \otimes X_1^{\beta_1}X_2^{\beta_2} \cdots X_d^{\beta_d}=0$, which is contradiction.
Thus $\ker \theta =(0)$.
Consequently, the map $\theta :D(-1) \to L$ is an isomorphism.
\end{proof}

Thanks to Lemma \ref{L1}, we can prove the following result.

\begin{prop}\label{function}
Suppose that $Q \cap I^2=QI$.
Then we have
\begin{eqnarray*}
\ell_A(A/I^{n+1})&=&\e_0(I)\binom{n+d}{d}-\{\e_0(I)-\ell_A(A/I)+\ell_A(I^2/QI)\}\binom{n+d-1}{d-1}\\
&& + \ell_A(I^2/QI)\binom{n+d-2}{d-2}-\ell_A(C_n)
\end{eqnarray*}
for all $n \geq 0$.
\end{prop}

\begin{proof}
We have, for all $n \geq 0$,
\begin{eqnarray*}
\ell_A(S_n) &=& \ell_A(L_n)+\ell_A(C_n)\\
&=& \ell_A(I^2/QI)\binom{n+d-2}{d-1}+\ell_A(C_n)\\
&=& \ell_A(I^2/QI)\binom{n+d-1}{d-1}-\ell_A(I^2/QI)\binom{n+d-2}{d-2}+\ell_A(C_n)
\end{eqnarray*}
by the exact sequence
$$ 0 \to L \to S \to C \to 0 \ \ \ (\dagger)$$
and the isomorphisms $L \cong D(-1) \cong (I^2/QI) \otimes_A (A/\Ann_A(I^2/QI))[X_1,X_2,\cdots,X_d]$ of graded $T$-modules (see Lemma \ref{L1}).
Therefore we have, for all $n \geq 0$,
\begin{eqnarray*}
\ell_A(A/I^{n+1})&=& \e_0(I)\binom{n+d}{d}-\{\e_0(I)-\ell_A(A/I)\}\binom{n+d-1}{d-1}-\ell_A(S_n)\\
&=& \e_0(I)\binom{n+d}{d}-\{\e_0(I)-\ell_A(A/I)\}\binom{n+d-1}{d-1}\\
&&- \left\{\ell_A(I^2/QI)\binom{n+d-1}{d-1}-\ell_A(I^2/QI)\binom{n+d-2}{d-2}+\ell_A(C_n) \right\}\\
&=& \e_0(I)\binom{n+d}{d}-\{\e_0(I)-\ell_A(A/I)+\ell_A(I^2/QI)\}\binom{n+d-1}{d-1}\\
&&+ \ell_A(I^2/QI)\binom{n+d-2}{d-2}-\ell_A(C_n)
\end{eqnarray*}
by \cite[Proposition 2.2 (2)]{GNO}.
\end{proof}

The following result specifies \cite[Proposition 2.2 (3)]{GNO} and,  by using Proposition \ref{Ass} and \ref{function},  the proof takes advantage of the same techniques. 

\begin{prop}\label{multi}
Suppose that $Q \cap I^2=QI$. Let $\p=\m T$.
Then we have $$\e_1(I)=\e_0(I)-\ell_A(A/I)+\ell_A(I^2/QI)+\ell_{T_{\p}}(C_{\p}).$$
\end{prop}

Combining Lemma \ref{fact1} (3) and Proposition \ref{multi} we obtain  the following result that was proven by Elias and Valla \cite[Theorem 2.1]{EV} in the case where $I=\m$.

\begin{cor}\label{EV}
Suppose that $Q \cap I^2=QI$.
Then we have $\e_1(I) \geq \e_0(I)-\ell_A(A/I)+\ell_A(I^2/QI)$.
The equality $\e_1(I) = \e_0(I)-\ell_A(A/I)+\ell_A(I^2/QI)$ holds true if and only if $I^3=QI^2$.
When this is the case, $\e_2(I)=\e_1(I)-\e_0(I)+\ell_A(A/I)$ if $d \geq 2$, $\e_i(I)=0$ for all $3 \leq i \leq d$, and $G$ is a Cohen-Macaulay ring.
\end{cor}

Let us introduce the relationship between the depth of the module $C$ and the associated graded ring $G$ of $I$.

\begin{lem}\label{depth}
Suppose that $Q \cap I^2=QI$ and $C \neq (0)$. 
Let $s=\depth_TC$.
Then we have $\depth G \geq s-1$.
In particular, we have $\depth G=s-1$, if $s \leq d-2$.
\end{lem}

\begin{proof}
Notice that $L \cong D(-1)$ by Lemma \ref{L1} and $D$ is a $d$ dimensional Cohen-Macaulay $T$-module.
Therefore, we have $s \leq \depth_TS$ and, if $s \leq d-2$ then $s=\depth_TS$ by the exact sequence
$$ 0 \to L \to S \to C \to 0 \ \ \ \ \ (\dagger).$$
Because $\depth G \geq \depth_TS-1$ and, if $\depth_TS \leq d-1$ then $\depth G = \depth_TS-1$ by \cite[Proposition 2.2 (4)]{GNO}, our assertions follow.
\end{proof}


\section{Proof of Theorem \ref{main1}}

The purpose of this section is to prove Theorem \ref{main1}.
Throughout this section, let   $I$ be an integrally closed $\m$-primary ideal.

\begin{thm}\label{rank1}
Suppose that $I$ is integrally closed. 
Then the following conditions  are equivalent: 
\begin{itemize}
\item[(1)] $\e_1(I)=\e_0(I)-\ell_A(A/I)+\ell_A(I^2/QI)+1$,
\item[(2)] $\m C=(0)$ and $\rank_B C=1$, 
\item[(3)] there exists a non-zero graded ideal $\a$ of $B$ such that $C \cong \a(-1)$ as graded $T$-modules.
\end{itemize}
\end{thm}

To prove Theorem \ref{rank1}, we need the following bound on $\e_2(I)$.

\begin{lem}\label{e2}$($\cite[Theorem 12]{I1}, $ $\cite[Corollary 2.5]{S2}, $ $\cite[Corollary 3.1]{RV3}, $) $
Suppose $d \geq 2$ and   let $I$ be  an integrally closed ideal,  then  $\e_2(I) \geq \e_1(I)-\e_0(I)+\ell_A(A/I). $
 
\end{lem}

\begin{proof}[Proof of Theorem \ref{rank1}]
Let $\p=\m T$ then we see that $\e_1(I)=\e_0(I)-\ell_A(A/I)+\ell_A(I^2/QI)+\ell_{T_{\p}}(C_{\p})$ by Lemma \ref{multi} and $\Ass_TC=\{ \p \}$ by Proposition \ref{Ass}.

$(1) \Rightarrow (2)$ 
Since $\ell_{T_{\p}}(C_{\p})=1$ and $\Ass_TC=\{ \p \}$, we have $\m C=(0)$ and $\rank_BC=1$. 

$(2) \Rightarrow (1)$
This is clear, because assertion $(1)$ is equivalent to saying that $\ell_{T_{\p}}(C_{\p})=1$.

$(3) \Rightarrow (2)$ This is obvious.

$(2) \Rightarrow (3)$
Because $\Ass_TC=\{ \p \}$, $C$ is torsion free $B$-module.
If $C$ is $B$-free, then we have $C \cong B(-2)$ as graded $B$-module, because $C_2 \neq (0)$ and $C_n=(0)$ for $n \leq 1$.
Hence $C \cong X_1B(-1)$ as graded $B$-modules with $0 \neq X_1 \in B_1$.

Suppose that $C$ is not $B$-free.
Then we have $d=\dim A \geq 2$.
Because $\rank_BC =1$, there exists a graded ideal $\a$ of $B$ such that $$C \cong \a(r)$$ as graded $B$-module for some $r \in \Z$.
Since every height one prime in the polynomial ring $B$ is principal, we may choose $\a$ with $\height_B \a \geq 2$.
Then since $\a_{r+2} \cong \a(r)_2 \cong C_2 \neq (0)$ and $\a_n=(0)$ for all $n \leq 0$, we have $r+2 \geq 1$ so that $r \geq -1$.
It is now enough to show that $r=-1$.
Applying the exact sequence
$$ 0 \to C \to B(r) \to (B/\a)(r) \to 0 $$
of graded $B$-modules, we have
\begin{eqnarray*}
\ell_A(C_n) &=& \ell_A(B_{r+n})-\ell_A([B/\a]_{r+n})\\
&=& \binom{n+r+d-1}{d-1}-\ell_A([B/\a]_{r+n})\\
&=& \binom{n+d-1}{d-1}+r\binom{n+d-2}{d-2}+\mbox{$($lower terms$)$}
\end{eqnarray*}
for all $n \gg 0$, because $\height_B\a \geq 2$.
Therefore, thanks to Proposition \ref{function}, we have
\begin{eqnarray*}
\ell_A(A/I^{n+1})&=&\e_0(I)\binom{n+d}{d}-\{\e_0(I)-\ell_A(A/I)+\ell_A(I^2/QI)\}\binom{n+d-1}{d-1}\\
&& + \ell_A(I^2/QI)\binom{n+d-2}{d-2}-\ell_A(C_n)\\
&=& \e_0(I)\binom{n+d}{d}-\{\e_0(I)-\ell_A(A/I)+\ell_A(I^2/QI)+1\}\binom{n+d-1}{d-1}\\
&& + \{\ell_A(I^2/QI)-r\}\binom{n+d-2}{d-2}+\mbox{$($lower terms$)$}
\end{eqnarray*}
for all $n \gg 0$.
Therefore $\e_2(I)=\ell_A(I^2/QI)-r$.
Then, because $\e_2(I) \geq \e_1(I)-\e_0(I)+\ell_A(A/I)=\ell_A(I^2/QI)+1$ by Lemma \ref{e2} we have $r \leq -1$.
Thus $r=-1$ and so $C \cong \a(-1)$ as graded $B$-modules.
\end{proof}

As a direct consequence of Theorem \ref{rank1} the following result holds true. 
\begin{prop}\label{function+}
Assume that $I$ is integrally closed.
Suppose that $\e_1(I)=\e_0(I)-\ell_A(A/I)+\ell_A(I^2/QI)+1$ and $I^4=QI^3$ and let $c=\ell_A(I^3/QI^2)$.
Then  
\begin{itemize}
\item[(1)] $1 \leq c \leq d$ and $\mu_B(C)=c$.
\item[(2)] $\depth ~G \geq d-c$ and $\depth_TC=d-c+1$,
\item[(3)] $\depth ~ G=d-c$, if $c \geq 2$.
\item[(4)] Suppose $c=1 <d$. Then  $HP_I(n) = \ell_A(A/I^{n+1}) $ for all $n \geq 0$ and 
\[ \e_i(I) =  \left\{
\begin{array}{ll}
\e_1(I)-\e_0(I)+\ell_A(A/I)+1 & \quad \mbox{if $i=2 $,} \\
1 & \quad \mbox{if $i =3$ and  $d \ge 3$,} \\
0 & \quad \mbox{if $4 \le i \le d.$}
\end{array}
\right.\]
\item[(5)] Suppose $2 \leq c < d$. Then  $HP_I(n) = \ell_A(A/I^{n+1}) $ for all $n \geq 0$ and 
\[ \e_i(I) =  \left\{
\begin{array}{ll}
\e_1(I)-\e_0(I)+\ell_A(A/I) & \quad \mbox{if $i=2 $,} \\
0 & \quad \mbox{if $i \neq c+1, c+2 $,\ \ $3 \leq i \leq d$} \\
(-1)^{c+1} & \quad \mbox{if $i =c+1, c+2$,\ \  $3 \leq i \leq d$}
\end{array}
\right.\]
 \item[(6)] Suppose $c=d$. Then  $HP_I(n) = \ell_A(A/I^{n+1}) $ for all $n \geq 2 $ and 
\[ \e_i(I) =  \left\{
\begin{array}{ll}
\e_1(I)-\e_0(I)+\ell_A(A/I) & \quad \mbox{if $i=2$ and $d \geq 2$,} \\
0 & \quad \mbox{if  $3 \leq i \leq d$}  
\end{array}
\right.\]
 
 \item[(7)] The Hilbert series $ HS_I(z)$   is given by {\small { $$ HS_I(z)=\frac{\ell_A(A/I)+\{\e_0(I)-\ell_A(A/I)-\ell_A(I^2/QI)-1\}z+\{\ell_A(I^2/QI)+1\}z^2+(1-z)^{c+1}z}{(1-z)^d}.$$}}
\end{itemize}
\end{prop}

\begin{proof}
$(1)$ We have $C=TC_2$, since $I^4=QI^3$ (c.f. Lemma \ref{fact1} (5)).
Therefore, thanks to Theorem \ref{rank1}, $C \cong \a(-1)$ as graded $T$-modules, where $\a=(X_1,X_2,\cdots,X_c)B$ is an ideal generated by linear forms $\{X_i\}_{1 \leq i \leq c}$ of $B$.
Hence, we get $1 \leq c \leq d$ and $\mu_B(C)=c$.

$(4)$, $(5)$, $(6)$ Let us consider the exact sequence
$$ 0 \to C \to B(-1) \to (B/\a)(-1) \to 0 \ \ \ \ (*_5) $$
of graded $B$-modules.
Then, we have 
\begin{eqnarray*}
\ell_A(C_n) &=& \ell_A(B_{n-1})-\ell_A([B/\a]_{n-1})\\
&=& \binom{n-1+d-1}{d-1}-\binom{n-1+d-c-1}{d-c-1}\\
&=& \binom{n+d-1}{d-1}-\binom{n+d-2}{d-2}-\binom{n+d-c-1}{d-c-1}+\binom{n+d-c-2}{d-c-2}
\end{eqnarray*}
for all $n \geq 0$ $($resp. $n\geq 2)$ if $1 \leq c \leq d-1$ $($resp. $c=d)$.
Therefore, our assertions $(4)$, $(5)$, and $(6)$ follow by Proposition \ref{function}.

(7) We have
$$ HS_C(z) =HS_B(z)z-HS_{B/\a} (z)z =\frac{z-(1-z)^cz}{(1-z)^d} $$
by the above exact sequence $(*_5)$, where $HS_{*} (z) $ denotes the Hilbert series of the graded modules.
We also have
$$HS_S(z)=HS_L(z)+ HS_C(z)=\frac{\{\ell_A(I^2/QI)+1\}z-(1-z)^{c}z}{(1-z)^d}$$
by the exact sequence
$$ 0 \to L \to S \to C \to 0  \ \ \ (\dag)$$
and isomorphisms $L \cong D(-1) \cong (I^2/QI)\otimes_A (A/\Ann_A(I^2/QI))[X_1,X_2,\cdots,X_d](-1)$ of graded $T$-modules $($Lemma \ref{L1}$)$.
Then, because
$$ HS_I(z) =\frac{\ell_A(A/I)+\{\e_0(I)-\ell_A(A/I)\}z}{(1-z)^d}-(1-z) HS_S(z) $$
$($\cite{VP}, \cite[Proposition 6.3]{RV3}$)$, we can get the required result.

We prove now $(2)$, $(3). $
We have $\depth_TC=d-c+1$ by the exact sequence $(*_5)$ so that $\depth G \geq d-c$ and, if $c \geq 3, $ then $\depth G=d-c$ by Lemma \ref{depth}.
Let us consider the case where $c=2$ and we need to show   $\depth G=d-2$.
Assume   $\depth G \geq d-1, $ then $S$ is a Cohen-Macaulay $T$-module by \cite[Proposition 2.2]{GNO}.
Taking the local cohomology functors $\H_M^i(*)$ of $T$ with respect to the graded maximal ideal $M=\m T+T_+$ to the above exact sequence $(\dag)$ of graded $T$-module,  we get  a  monomonophism $$\H_{M}^{d-1}(C) \hookrightarrow \H_{M}^d(L)$$ of graded $T$-module.
Because $C \cong (X_1,X_2)B(-1)$, we have $\H_{M}^{d-1}(C) \cong \H_{M}^{d-2}(B/(X_1,X_2)B)(-1)$ as graded $T$-modules so that $[\H_{M}^{d-1}(C)]_{-d+3} \neq (0)$ $($notice that $B/(X_1,X_2)B \cong (A/\m)[X_3,X_4,\cdots,X_d]$$)$.
On the other hand, we have $[\H_M^d(L)]_{n}=(0)$ for all $n \geq -d+2$, because $L \cong D(-1) \cong (I^2/QI) \otimes_A (A/\Ann_A(I^2/QI))[X_1,X_2,\cdots,X_d](-1)$ by Lemma \ref{L1}.
However, it is impossible.
Therefore, $\depth~ G=d-2$ if $c=2$.
\end{proof}

We prove now  Theorem \ref{main1}.
Assume assertion $(1)$ in Theorem \ref{main1}. 
Then we have an isomorphism $C \cong \a(-1)$ as graded $B$-modules for a graded ideal $\a$ in $B$ by Theorem \ref{rank1}.
Once we are able to show $I^4=QI^3$, then, because $C=TC_2$ by Lemma \ref{fact1} (5), the ideal $\a$ is generated by linearly independent linear forms $\{X_i\}_{1 \leq i \leq c}$ of $B$ with $c=\ell_A(I^3/QI^2)$ $($recall that $\a_1 \cong C_2 \cong I^3/QI^2$ by Lemma \ref{fact1} $(2)$$)$.
Therefore, the implication $(1) \Rightarrow (3)$ in Theorem \ref{main1} follows.
We also notice that, the last assertions of Theorem \ref{main1} follow by Proposition \ref{function+}.

Thus our Theorem \ref{main1} has been proven modulo the following theorem.

\begin{thm}\label{m4}
Assume that $I$ is integrally closed.
Suppose that $\e_1(I)=\e_0(I)-\ell_A(A/I)+\ell_A(I^2/QI)+1$. Then $I^4=QI^3$.
\end{thm}

\begin{proof}
We proceed by induction on $d$.
Suppose that $d=1$. Then the result follows by  \cite[Proposition 4.6]{HM} since $\e_1(I) = \sum_{i \ge 0} \ell_A(I^{i+1}/QI^i). $
 
Assume that $d \geq 2$ and that our assertion holds true for $d-1$.
Since the residue class field $A/\m$ of $A$ is infinite, without loss of generality, we may assume that $a_1$ is superficial element of $I$ and $I/(a_1)$ is integrally closed $($c.f. \cite[page. 648]{I1}, \cite[Proposition 1.1]{RV3}$)$.
We set $A'=A/(a_1), \ I'=I/(a_1), \ Q'=Q/(a_1)$.
We then have $\e_1(I')=\e_0(I')-\ell_{A'}(A'/I')+\ell_{A'}(I'^2/Q'I')+1$
because $\e_i(I')=\e_i(I)$ for $0 \leq i \leq d-1$, $\ell_{A'}(A'/I')=\ell_A(A/I)$, and since $(a_1) \cap I^2=a_1I$, $\ell_{A'}(I'^2/Q'I')=\ell_A(I^2/QI)$.
Then the inductive assumption on $d$ says that $I'^4=Q'I'^3$ holds true.
If $\depth \rmG(I')>0$ then, thanks to Sally's technique (c.f. \cite{S1}, \cite[Lemma 2.2]{HM}), $a_1t$ is  a non-zero divisor on $G$. 
Then we have $I^4=QI^3$.

Assume that $\depth \rmG(I')=0$.
Then because $\e_1(I')=\e_0(I')-\ell_{A'}(A'/I')+\ell_{A'}(I'^2/Q'I')+1$ and $I'^4=Q'I'^3$, we have $\ell_{A'}({I'}^3/{Q'}{I'}^2)=d-1$ by Proposition \ref{function+} (2).
Since ${I'}^3/{Q'}{I'}^2$ is a homomorphic image of $I^3/QI^2$, we have $\ell_A(C_2)=\ell_A(I^3/QI^2) \geq d-1$.
Let us now take an isomorphism
$$ \varphi:C \to \a(-1) $$
of graded $B$-modules, where $\a$ is a graded ideal of $B$ $($c.f. Theorem \ref{rank1}$)$.
Then since $\ell_A(\a_1)=\ell_A(C_2) \geq d-1$, we have $Y_1,Y_2,\cdots,Y_{d-1} \in \a$ where $Y_1,Y_2,\cdots,Y_{d-1}$ denote linearly independent linear forms of $B$, which we enlarge to a basis $Y_1,Y_2,\cdots,Y_{d-1}, Y_d$ of $B_1$.
If $\a=(Y_1,Y_2,\cdots,Y_{d-1})B$ then, because $C=TC_2$, we have $I^4=QI^3$ by Lemma \ref{fact1} (5).
Assume that $\a \neq (Y_1,Y_2,\cdots,Y_{d-1})B$.
Then, since $B=k[Y_1,Y_2,\cdots,Y_{d}]$, the ideal $\a/(Y_1,Y_2,\cdots,Y_{d-1})$ is principal so that $\a=(Y_1,Y_2,\cdots,Y_{d-1}, Y_d^{\ell})B$ for some $\ell \geq 1$.

We need the following.

\begin{claim}\label{claim1}
We have $\ell=1$ or $\ell= 2$.
\end{claim}

\begin{proof}[Proof of Claim \ref{claim1}]
Assume that $\ell \geq 3$.
Then $I^4/[QI^3+\m I^4] \cong [C/MC]_3=(0)$, where $M=\m T+T_+$ denotes the graded maximal ideal of $T$.
Therefore $I^4=QI^3$ by Nakayama's lemma so that $C=B C_2$, which is impossible.
Thus $\ell=1$ or $\ell= 2$.
\end{proof}

We have to show that $\ell=1$.
Assume that $\ell=2$.
Let us write, for each $1 \leq i \leq d$, $Y_i=\overline{b_it}$ with $b_i \in Q$, where $\overline{b_it}$ denotes the image of $b_it \in T$ in $B$.
We notice here that $Q=(b_1,b_2,\cdots,b_d)$, because $Y_1,Y_2,\cdots,Y_{d}$ forms a $k$-basis of $B_1$.

Let us choose elements $f_i \in C_2$ for $1 \leq i \leq d-1$ and $f_d \in C_3$ so that $\varphi(f_i)=Y_i$ for $1 \leq i \leq d-1$ and $\varphi(f_d)=Y_d^2$.
Let $z_i \in I^3$ for $1 \leq i \leq d-1$ and $z_d \in I^4$ so that $\{f_i\}_{1 \leq i \leq d-1}$ and $f_d$ are, respectively, the images of $\{z_it^2\}_{1 \leq i \leq d-1}$ and $z_dt^3$ in $C$.
Let us now consider the relations
$Y_d^2f_i=Y_if_d$
in $C$ for $1 \leq i \leq d-1$, that is
$$b_d^2z_i-b_iz_d \in Q^3I^2$$
for $1 \leq i \leq d-1$.
Notice that $$ (b_1, b_d) \cap Q^3 I^2=(b_1,b_d)Q^2 I^2$$ by Lemma \ref{cap}
and write
$$ b_d^2z_1-b_1z_d=b_1\tau_1+b_d\tau_d  \ \ \ \ \ (1)$$
with $\tau_1,\tau_d \in Q^2I^2$.
Then we have $$b_d(b_dz_1-\tau_d)=b_1(z_d+\tau_1)$$
so that $ b_dz_1-\tau_d \in (b_1)$ because $b_1,b_d$ forms a regular sequence on $A$.
Since $\tau_d \in (b_1,b_d) \cap Q^2I^2=(b_1,b_d)QI^2$ by Lemma \ref{cap}, there exist elements $\tau_1', \tau_d' \in Q I^2$ such that $\tau_d=b_1 \tau_1'+b_d \tau_d'$.
Then   by the equality $(1)$ we have 
$$ b_d^2(z_1-\tau'_d)=b_1(z_d+\tau_1+b_d\tau'_1) \ \ \ \ \ (2)$$
so that we have $z_1-\tau_d' \in (b_1)$.
Hence $z_1 \in QI^2+(b_1)$.
The same argument works for each $1 \leq i \leq d-1$ to see $z_i \in QI^2+(b_i)$.
Therefore, because $I^3 = QI^2+(z_1,z_2,\cdots,z_{d-1})$, we have $I^3 \subseteq b_dI^2+(b_1,b_2,\cdots,b_{d-1})$ and hence
$$ I^4 \subseteq b_d^2I^2+ (b_1,b_2,\cdots,b_{d-1}).$$
Then, because $z_d+\tau_1+b_d\tau'_1 \in I^4$, there exist elements $h \in I^2$ and $\eta \in (b_1,b_2,\cdots,b_{d-1})$ such that $z_d+\tau_1+b_d\tau'_1=b_d^2h+\eta$.
Then we have $$b_d^2(z_1-\tau'_d -b_1h)=b_1 \eta \ \ \ \ \ (3) $$ by the equality $(2)$.
Since $b_1,b_2,\cdots,b_d$ is a regular sequence on $A$, $\eta \in (b_d^2) \cap (b_1,b_2,\cdots, b_{d-1})=b_d^2(b_1,b_2,\cdots, b_{d-1})$.
Write $\eta =b_d^2\eta'$ with $\eta' \in (b_1,b_2,\cdots, b_{d-1})$, then we have
$$ b_d^2(z_1-\tau'_d -b_1h)=b_1b_d^2\eta'$$
by the equality (3), so that $z_1-\tau'_d -b_1h=b_1\eta'$.
Then we have $b_1\eta'=z_1-\tau'_d -b_1h \in Q^2 \cap I^3=Q^2I$, since $Q^2 \cap I^3 \subseteq Q^2 \cap \overline{Q^3}=Q^2 \overline{Q}=Q^2I$ (c.f. \cite{H,I2}), where $\overline{J}$ denotes the integral closure of an ideal $J$.
Hence $z_1 \in QI^2$, because $\tau'_1, b_1h, b_1\eta' \in QI^2$.
Therefore $f_1=0$ in $C$, which is impossible.
Thus $\ell=1$ so that we have $\a=(X_1,X_2,\cdots,X_d)B$.
Therefore $C=TC_2$, that is $I^4=QI^3$.
This completes the proof of Theorem \ref{m4} and that of Theorem \ref{main1} as well.
\end{proof}

\section{Consequences}

The purpose of this section is to present  some consequences of Theorem \ref{main1}.
Let us begin with the following which is exactly the case where $c=1$ in Theorem \ref{main1}.

\begin{thm}\label{B(-2)}
Assume that $I$ is integrally closed. Then the following conditions are equivalent.
\begin{itemize}
\item[(1)] $C \cong B(-2)$ as graded $T$-modules.
\item[(2)] $\e_1(I)=\e_0(I)-\ell_A(A/I)+\ell_A(I^2/IQ)+1$, and if $d \geq 2$ then $\e_2(I) \neq \e_1(I)-\e_0(I)+\ell_A(A/I)$.
\item[(3)] $\ell_A(I^3/QI^2)=1$ and $I^4=QI^3$.
\end{itemize}
When this is the case, the following assertions follow.
\begin{itemize}
\item[(i)] $\depth G \geq d-1$.
\item[(ii)] $\e_2(I)=\e_1(I)-\e_0(I)+\ell_A(A/I)+1$ if $d \geq 2$.
\item[(iii)] $\e_3(I)=1$ if $d \geq 3$, and $\e_i(I)=0$ for $4 \leq i \leq d$.
\item[(iv)] The Hilbert series $HS_I(z) $  is given by
$$ HS_I(z) =\frac{\ell_A(A/I)+\{\e_0(I)-\ell_A(A/I)-\ell_A(I^2/QI)\}z+\{\ell_A(I^2/QI)-1\}z^2+z^3}{(1-z)^d}. $$
\end{itemize}
\end{thm}

\begin{proof}
For $(1) \Leftrightarrow (2)$, $(1) \Rightarrow (3)$ and the last assertions see Theorem \ref{main1} with $c=1$.

$(3) \Rightarrow (1)$:
Since $\m I^{n+1}=Q^{n-1}I^2$ for all $n \geq 2$, we have $\m C=(0)$ by Lemma \ref{fact1} (4).
Then we have an epimorphism $B(-2) \to C \to 0$ of graded $T$-modules, which must be an isomorphism because $\dim_TC=d$ $($Proposition \ref{Ass}$)$.
\end{proof}

Let $\widetilde{J}=\bigcup_{n \geq 1}[J^{n+1}:_AJ^n]=\bigcup_{n \geq 1} J^{n+1} :_A (a_1^n,a_2^n,\cdots,a_d^n)$ denote the Ratliff-Rush closure of an $\m$-primary ideal $J$ in $A$, which is the largest $\m$-primary ideal in $A$ such that $J \subseteq \widetilde{J}$ and $\e_i(\widetilde{J})=\e_i(J)$ for all $0 \leq i \leq d$ $($c.f. \cite{RR}$)$.

Let us note the following remark.

\begin{rem}\label{tilde}
Assume that $I$ is integrally closed.
Then, by \cite{H, I1}, $Q^n \cap \overline{I^{n+1}}=Q^n \cap \overline{Q^{n+1}}=Q^n\overline{Q}=Q^nI$ holds true for $n \geq 1$, where $\overline{J}$ denotes the integral closure of an ideal $J$.
Thus, we have $Q^n \cap \widetilde{I^{n+1}}=Q^nI$ for all $n \geq 1$.
\end{rem}

The following result correspond to the case where $c=d$ in Theorem \ref{main1}.
In section 5 we give an example of the maximal ideal which satisfy assertion $(1)$ in Theorem \ref{B_+}.

\begin{thm}\label{B_+}
Suppose that $d \geq 2$ and assume that $I$ is integrally closed.
Then the following  conditions are equivalent: 
\begin{itemize}
\item[(1)] $C \cong B_+(-1)$ as graded $T$-modules.
\item[(2)] $\e_1(I)=\e_0(I)-\ell_A(A/I)+\ell_A(I^2/QI)+1$, $\e_2(I)=\e_1(I)-\e_0(I)+\ell_A(A/I)$, and $\e_i(I)=0$ for all $3 \leq i \leq d$.
\item[(3)] $\ell_A(\widetilde{I^2}/I^2)=1$ and $\widetilde{I^{n+1}}=Q\widetilde{I^n}$ for all $n \geq 2$.
\end{itemize}
When this is the case, the associated graded ring $G$ of $I$ is a Buchsbaum ring with $\depth ~G=0$ and the Buchsbaum invariant $\Bbb{I}(G)=d$.
\end{thm}

\begin{proof}
We set $c=\ell_A(I^3/QI^2)$ and $\calF=\{\widetilde{I^n}\}_{n \geq 0}$.
Let $\rmR'(\calF)=\sum_{n \in \Z}\widetilde{I^n}t \subseteq A[t,t^{-1}]$ and $\rmG(\calF)=\rmR'(\calF)/t^{-1}\rmR'(\calF)$.
Let $\e_i(\calF)$ denote the $i$th Hilbert coefficients of the filtration $\calF$ for $0 \leq i \leq d$.

$(1) \Rightarrow (2)$
  follows from Theorem \ref{main1}, because $c=\ell_A(C_2)=d$.

$(2) \Rightarrow (1)$
Because $\e_2(I)=\e_1(I)-\e_0(I)+\ell_A(A/I)$ and $\e_i(I)=0$ for all $3 \leq i \leq d$, we have $c=d$ by Theorem \ref{main1}.
Therefore $C \cong B_+(-1)$ as graded $T$-modules.

$(1) \Rightarrow (3)$
Since $c=d$, we have $\depth G=0$ by Theorem \ref{main1} $(ii)$.
We apply local cohomology functors $\H_{M}^i(*)$ of $T$ with respect to the graded maximal ideal $M=\m T+T_+$ of $T$ to the exact sequences
$$ 0 \to I^2 R(-1) \to I R(-1) \to G_+ \to 0 \ \ \mbox{and} \ \ 0 \to I^2 T(-1) \to I^2R(-1) \to C \to 0 $$
of graded $T$-modules and get derived monomorphisms
$$  \H_{M}^0(G_+) \hookrightarrow \H_{M}^1(I^2 R)(-1) \ \ \mbox{and} \ \ \H_M^1(I^2 R)(-1)  \hookrightarrow \H_M^1(C) $$
because $\depth_T IR >0$ and $\depth_T I^2 T \geq 2$ (recall that $T$ is a Cohen-Macaulay ring with $\dim T=d+1$ and $\depth_T T/I^2T \geq 1$ by Lemma \ref{cap}).
We furthermore have $\H_M^1(C) \cong (B/B_+)(-1)$ since $C \cong B_+(-1)$.
Since $I$ is integrally closed, we have $[\H_M^0(G)]_0=(0)$ so that $\H_M^0(G) \cong \H_M^0(G_+) \neq (0)$.
Then because $\ell_A(B/B_+)=1$, we have isomorphisms
$$ \H_M^0(G) \cong \H_M^1(I^2 R)(-1) \cong \H_M^1(C) \cong B/B_+(-1) $$
of graded $B$-modules and hence $\H_M^0(G)=[\H_M^0(G)]_1 \cong B/B_+$.
Then since $[\H_M^0(G)]_1 \cong \widetilde{I^2}/I^2$ we have $\ell_A(\widetilde{I^2}/I^2)=1$.
Hence we have
\begin{eqnarray*}
\e_1(\mathcal{F})=\e_1(I) &=& \e_0(I)-\ell_A(A/I)+\ell_A(I^2/QI)+1\\
&=& \e_0(\mathcal{F})-\ell_A(A/I)+\ell_A(\widetilde{I^2}/QI)+1-\ell_A(\widetilde{I^2}/I^2)\\
&=& \e_0(\mathcal{F})-\ell_A(A/I)+\ell_A(\widetilde{I^2}/Q \cap \widetilde{I^2})
\end{eqnarray*}
because $\widetilde{I}=I$, $Q \cap \widetilde{I^2}=QI$ by Remark \ref{tilde}, and $\e_i(\mathcal{F})=\e_i(I)$ for $i=0,1$.
Therefore, $\widetilde{I^{n+1}}=Q\widetilde{I^n}$ for all $n \geq 2$ by \cite[Theorem 2.2]{GR}.

$(3) \Rightarrow (2)$
Because $Q \cap \widetilde{I^2}=QI$ by Remark \ref{tilde} and $\widetilde{I^{n+1}}=Q\widetilde{I^n}$ for all $n \geq 2$, we have $\e_1(\mathcal{F})= \e_0(\mathcal{F})-\ell_A(A/I)+\ell_A(\widetilde{I^2}/Q I)$ and $\rmG(\calF)$ is a Cohen-Macaulay ring by \cite[Theorem 2.2]{GR}.
Then, since $\ell_A(\widetilde{I^2}/I^2)=1$, we have $\e_1(I)=\e_0(I)-\ell_A(A/I)+\ell_A(I^2/QI)+1$.
We furthermore have $\e_2(I)=\e_2(\calF)=\ell_A(\widetilde{I^2}/QI)=\ell_A(I^2/QI)+1=\e_1(I)-\e_0(I)+\ell_A(A/I)$, and $\e_i(I)=\e_i(\mathcal{F})=0$ for $3 \leq i \leq d$, because $\rmG(\calF)$ is a Cohen-Macaulay ring $($c.f. \cite[Proposition 4.6]{HM}$)$.

Assume one of the equivalent conditions. 
We have $\H_{M}^0(G)=[\H_{M}^0(G)]_1$ by the proof of the implication $(1) \Rightarrow (3)$.
Let $n \geq 3$ be an integer.
We then have $$ \widetilde{I^n}/I^n =\widetilde{I^n} \cap I^{n-1}/I^n \cong [\H_{M}^0(G)]_{n-1}=(0) $$
because $\widetilde{I^n}=Q^{n-2}\widetilde{I^{2}} \subseteq I^{n-1}$.
Therefore, we have $\widetilde{I^n}=I^n$ for all $n \geq 3$.

We set $W=\rmR'(\mathcal{F})/R'$ and look at the exact sequence
$$ 0 \to R' \to \rmR'(\mathcal{F}) \to W \to 0 \ \ \ (*') $$
of graded $R'$-modules.
Since $\widetilde{I^n}=I^n$ for all $n \neq 2$, we have $W=W_2=\widetilde{I^2}/I^2$ so that $\ell_A(W)=1$.
Then, because $\rmG(\mathcal{F})=\rmR'(\calF)/t^{-1}\rmR'(\calF)$ is a Cohen-Macaulay ring, so is $\rmR'(\mathcal{F})$.
Let $M'=(\m, R_+, t^{-1})R'$ be the unique graded maximal ideal in $R'$.
Then applying local cohomology functors $\H_{M'}^i(*)$ to the exact sequence $(*')$ yields $\H_{M'}^i(R')=(0)$ for all $i \neq 1, d+1$ and $\H_{M'}^1(R')=W$.
Since $\m W=(0)$, we have $\m \H_{M'}^1(R')=(0)$.
Thus, $R'$ is a Buchsbaum ring with the Buchsbaum invariant $$ \Bbb{I}(R')=\sum_{i=0}^{d}\binom{d}{i}\ell_A(\H_{M'}^i(R'))=d $$
and hence so is the graded ring $G=R'/t^{-1}R'$.
This completes the proof of Theorem \ref{B_+}.
\end{proof}

In the rest of this section, we explore the relationship between the inequality of Northcott \cite{N} and the structure of the graded module $C$ of an integrally closed ideal.

It is well known that  the inequality $\e_1(I) \geq \e_0(I)-\ell_A(A/I)$ holds true $($\cite{N}$)$ and the equality holds  if and only if $I^2=QI$ $($\cite[Theorem 2.1]{H}$)$.
When this is the case, the associated graded ring $G$ of $I$ is Cohen-Macaulay.

Suppose that $I$ is integrally closed and $\e_1(I)=\e_0(I)-\ell_A(A/I)+1$ then, thanks to \cite[Corollary 14]{I1}, we have $I^3=QI^2$ and the associated graded ring $G$ of $I$ is Cohen-Macaulay.
Thus the integrally closed ideal $I$ with $\e_1(I) \leq \e_0(I)-\ell_A(A/I)+1$ seems satisfactory understood.
In this section, we briefly study the integrally closed ideals $I$ with $\e_1(I)=\e_0(I)-\ell_A(A/I)+2$, and $\e_1(I)=\e_0(I)-\ell_A(A/I)+3$.

Let us begin with the following.

\begin{thm}\label{e_0+1}
Assume that $I$ is integrally closed.
Suppose that $\e_1(I)=\e_0(I)-\ell_A(A/I)+2$ and $I^3 \neq QI^2$.
Then the following assertions hold true.
\begin{itemize}
\item[(1)] $\ell_A(I^2/QI)=\ell_A(I^3/QI^2)=1$, and $I^4=QI^3$.
\item[(2)] $C \cong B(-2)$ as graded $T$-modules.
\item[(3)] $\depth ~G=d-1$.
\item[(4)] $\e_2(I)=3$ if $d \geq 2$, $\e_3(I)=1$ if $d \geq 3$, and $\e_i(I)=0$ for $4 \leq i \leq d$.
\item[(5)] The Hilbert series $HS_I(z) $   is given by
$$ HS_I(z)=\frac{\ell_A(A/I)+\{\e_0(I)-\ell_A(A/I)-1\}z+z^3}{(1-z)^d}. $$
\end{itemize}
\end{thm}

\begin{proof}
Because $I^3 \neq QI^2$, it follows from Corollary \ref{EV} that
$$ 0 < \ell_A(I^2/QI)< \e_1(I)-\e_0(I)+\ell_A(A/I)=2.$$
Therefore, $\ell_A(I^2/QI)=1$ and hence $\e_1(I)=\e_0(I)-\ell_A(A/I)+\ell_A(I^2/QI)+1$.
Let $I^2=QI+(xy)$ with $x$, $y \in I \backslash Q$.
Then $I^3=QI^2+(x^2y)$, so that $\ell_A(I^3/QI^2)=1$ since $I^3 \neq QI^2$ and $\m I^2 \subseteq QI$.
Thanks to Theorem \ref{main1}, $C \cong B(-2)$ as graded $T$-modules, so that assertions $(1)$, $(2)$, $(4)$, and $(5)$ follow, and $\depth ~G \geq d-1$ by Theorem \ref{B(-2)}.
Since $I^3 \subseteq Q I$, $G$ is not a Cohen-Macaulay ring, for otherwise $I^3=Q \cap I^3 =QI^2$, so that $\depth G=d-1$.
This completes the  proof of Theorem \ref{e_0+1}.
\end{proof}

Notice that the following result  also follows by \cite[Theorem 4.6]{RV3}.

\begin{cor}\label{e_0+1_}
Assume that $I$ is integrally closed and suppose that $\e_1(I)=\e_0(I)-\ell_A(A/I)+2$.
Then $\depth~ G \geq d-1$ and $I^4=QI^3$, and the graded ring $G$ is Cohen-Macaulay if and only if $I^3=Q I^2$.
\end{cor}

Before closing this section, we briefly study the integrally closed ideal $I$ with $\e_1(I)=\e_0(I)-\ell_A(A/I)+3$.
Suppose that $\e_1(I)=\e_0(I)-\ell_A(A/I)+3$ then we have
$$0< \ell_A(I^2/QI) \leq \e_1(I)-\e_0(I)+\ell_A(A/I) =3$$
by Corollary \ref{EV}.
If $\ell_A(I^2/QI)=1$ then we have $\depth G \geq d-1$ by \cite{RV1, W}.
If $\ell_A(I^2/QI)=3$ then the equality $\e_1(I)=\e_0(I)-\ell_A(A/I)+\ell_A(I^2/QI)$ holds true, so that $I^3=QI^2$ and the associated graded ring $G$ of $I$ is Cohen-Macaulay by Corollary \ref{EV}.
Thus we need to consider the following.

\begin{thm}\label{e_0+2}
Suppose that $d \geq 2$. Assume that $I$ is integrally closed and $\e_1(I)=\e_0(I)-\ell_A(A/I)+3$ and $\ell_A(I^2/QI)=2$.
Let $c=\ell_A(I^3/QI^2)$.
Then the following assertions hold true.
\begin{itemize}
\item[(1)] Either $C \cong B(-2)$ as graded $T$-modules or there exists an exact sequence $$0 \to B(-3) \to B(-2) \oplus B(-2) \to C \to 0$$ of graded $T$-modules.
\item[(2)] $ 1 \leq c \leq 2$ and $I^4=QI^3$.
\item[(3)] Suppose $c=1$ then $\depth G \geq d-1$ and $\e_2(I)=4$, $\e_3(I)=1$ if $d \geq 3$, and $\e_i(I)=0$ for $4 \leq i \leq d$.
\item[(4)] Suppose $c=2$ then $\depth G=d-2$ and $\e_2(I)=3$, $\e_3(I)=-1$ if $d \geq 3$, $\e_4(I)=-1$ if $d \geq 4$, and $\e_i(I)=0$ for $5 \leq i \leq d$.
\item[(5)] The Hilbert series $HS_I(z) $  is given by 
\[ HS_I(z) =  \left\{
\begin{array}{rl}
\vspace{3mm}
\displaystyle\frac{\ell_A(A/I)+\{\e_0(I)-\ell_A(A/I)-2\}z+z^2+z^3}{(1-z)^d}, & \quad \mbox{if $c=1$,} \\
\displaystyle\frac{\ell_A(A/I)+\{\e_0(I)-\ell_A(A/I)-2\}z+3z^3-z^4}{(1-z)^d} & \quad \mbox{if $c=2$}.
\end{array}
\right.\]
\end{itemize}
\end{thm}

\begin{proof}
Since $\ell_A(I^2/QI)=2$ and $\e_1(I)=\e_0(I)-\ell_A(A/I)+3$, we have $\e_1(I)=\e_0(I)-\ell_A(A/I)+\ell_A(I^2/QI)+1$.
We also have $1 \leq \ell_A(I^3/QI^2) \leq 2$ (see the proof of \cite[Proposition 2.1 and 2.2]{RV2}).
Then, thanks to Theorem \ref{main1}, $C \cong X_1B(-1)$ or $C \cong (X_1,X_2)B(-1)$ as graded $T$-modules, where $X_1$, $X_2$ denote the linearly independent linear forms of $B$.
Thus all assertions follow by Theorem \ref{main1}.
\end{proof}

We remark that $\ell_A(I^2/QI)$ measures how far is the multiplicity of $I$ from the minimal value, see \cite[Corollary 2.1]{RV3}. If $\ell_A(I^2/QI) \le 1,  $ then  $\depth ~G \geq d-1,  $ but  it is still open the problem whether $\depth G \geq d-2,  $ assuming $\ell_A(I^2/QI)=2$.
Theorem \ref{e_0+2} confirms the conjectured bound.

\begin{cor}\label{e_0+2_}
Assume that $I$ is integrally closed. 
Suppose that $\e_1(I)=\e_0(I)-\ell_A(A/I)+3$. 
Then $\depth~ G \geq d-2$.
\end{cor}

\begin{proof}
We have
$0< \ell_A(I^2/QI) \leq \e_1(I)-\e_0(I)+\ell_A(A/I) =3$
by Corollary \ref{EV}.
If $\ell_A(I^2/QI)=1$ or $\ell_A(I^2/QI)=3$ then we have $\depth G \geq d-1$ as above.
Suppose that $\ell_A(I^2/QI)=2$ then we have $\depth G \geq d-2$ by Theorem \ref{e_0+2} (3), (4).
This completes a proof of Corollary \ref{e_0+2_}.
\end{proof}


\section{An Example}

The goal of this section is to   construct an example of Cohen-Macaulay local ring with the maximal ideal $\m$  satisfying  the equality in Theorem \ref{main1} (1).
The class of examples we exhibit includes an interesting example given by H.-J. Wang, see \cite[Example 3.2]{RV3}.

\begin{thm}\label{ex1}
Let $d \geq c \geq 1$ be integers.
Then there exists a Cohen-Macaulay local ring $(A, \m)$ such that
$$ d=\dim A, \ \ \e_1(\m)=\e_0(\m)+\ell_A(\m^2/Q\m), \ \ \mbox{and} \ \ c=\ell_A(\m^3/Q\m^2) $$
for some minimal reduction $Q=(a_1,a_2,\cdots,a_d)$ of $\m$.
\end{thm}

To construct necessary examples we may assume that $c = d$. 
In fact, suppose that $0 < c < d$ and assume that we have already chosen a certain Cohen-Macaulay local ring $(A_0, \fkm_0)$ such that $c = \operatorname{dim}A_0, \ \ \e_1(\m_0) = \e_0(\m_0) + \ell_{A_0}(\m_0^2/Q_0 \m_0)$, and $c = \ell_{A_0}(\m_0^3/Q_0\m_0^2)$ with $Q_0 = (a_1, a_2, \cdots, a_c)A_0$ a minimal reduction of $\m_0$.
Let $n = d - c$ and let $A = A_0[[X_1, X_2, \cdots, X_{n}]]$ be the formal power series ring over the ring $A_0$. 
We set $\m = \m_0A + (X_1, X_2, \cdots, X_{n})A$ and $Q = Q_0A + (X_1, X_2, \cdots, X_{n})A$. 
Then $A$ is a Cohen-Macaulay local ring with $\operatorname{dim} A = d$ and the maximal ideal $\fkm = \fkm_0A + (X_1, X_2, \cdots, X_{n})A$. The ideal $Q$ is a reduction of $\m$ and because $X_1, X_2, \cdots, X_n$ forms a super regular sequence in $A$ with respect to $\m$ (recall that ${\rmG} (\m) = {\rmG} (\m_0)[Y_1, Y_2, \cdots, Y_n]$ is the polynomial ring, where $Y_i$'s are the initial forms of $X_i$'s), we have $\e_i(\m) = \e_i(\m_0)$ for $i = 0, 1$, $\m^2/Q\m \cong \m_0^2/Q_0\m_0$, and $\m^3/Q\m^2 \cong \m_0^3/Q_0\m_0^2$.
Thus we have $\e_1(\m) = \e_0(\m) + \ell_A(\m^2/Q\m)$ and $\ell_A(\m^3/Q\m^2 ) = c$. 
This observation allows us to concentrate our attention on the case where $c = d$. 

Let $m \geq 0$ and $d \geq 1$ be integers.
Let $$D=k[[\{X_j\}_{1 \leq j \leq m}, Y, \{V_i\}_{1 \leq i \leq d}, \{Z_i\}_{1 \leq i \leq d}]]$$
be the formal power series ring with $m+2d+1$ indeterminates over an infinite field $k$, and let
$$ \a=[(X_j \ | \ 1 \leq j \leq m)+(Y)]\cdot[(X_j \ | \ 1 \leq j \leq m)+(Y)+(V_i \ | \ 1 \leq i \leq d)] $$
$$+(V_iV_j \ | \ 1 \leq i,j \leq d, \  i \neq j)+(V_i^3-Z_iY \ | \ 1 \leq i \leq d).$$
We set $A=D/\a$ and denote the images of $X_j$, $Y$, $V_i$, and $Z_i$ in $A$ by $x_j$, $y$, $v_i$, and $a_i$, respectively.
Then, since $\sqrt{\a}=(X_j \ | \ 1 \leq j \leq m)+(Y)+(V_i \ | \ 1 \leq i \leq d)$, we have $\dim A=d$.
Let $\m=(x_j \ | \ 1 \leq j \leq m)+(y)+(v_i \ | \ 1 \leq i \leq d)+(a_i \ | \ 1 \leq i \leq d)$ be the maximal ideal in $A$ and we set $Q=(a_i \ | \ 1 \leq i \leq d)$.
Then, $\m^2=Q\m+(v_i^2 \ | \ 1 \leq i \leq d)$, $\m^3=Q\m^2+Qy$, and $\m^4=Q\m^3$.
Therefore $Q$ is a minimal reduction of $\m$, and $a_1,a_2,\cdots,a_d$ is a system of parameters for $A$.

We are now interested in the Hilbert coefficients $\e_i(\m)'$ of the maximal ideal $\m$ as well as the structure of the associated graded ring $\rmG(\m)$ and the module $\rmC_Q(\m)$ of $\m$.

\begin{thm}\label{ex2}
The following assertions hold true.
\begin{itemize}
\item[(1)] $A$ is a Cohen-Macaulay local ring with $\dim A=d$.
\item[(2)] $\rmC_Q(\m) \cong B_+(-1)$ as graded $T$-modules. Therefore, $\ell_A(\m^3/Q\m^2)=d$.
\item[(3)] $\e_0(\m)=m+2d+2$, $\e_1(\m)=m+3d+2$.
\item[(4)] $\e_2(\m)=d+1$ if $d \geq 2$, and $\e_i(\m)=0$ for all $3 \leq i \leq d$.
\item[(5)] $\rmG(\m)$ is a Buchsbaum ring with $\depth \rmG(\m)=0$ and ${\Bbb I}(\rmG(\m))=d$.
\item[(6)] The Hilbert series $HS_{\m}(z) $ of $A$  is given by
$$ HS_{\m}(z)  =\frac{1+\{m+d+1\}z+\sum_{j=3}^{d+2}(-1)^{j-1}\binom{d+1}{j-1}z^{j}}{(1-z)^d}. $$
\end{itemize}
\end{thm}

\vskip 2mm
Notice that Wang's example before quoted corresponds to the particular case $m=0 $ and $d=2.$  
\vskip 2mm
Let us divide the proof of Theorem \ref{ex2} into two steps.
Let us begin with the following.

\begin{prop}\label{ex3}
Let $\p=(X_j \ | \ 1 \leq j \leq m)+(Y)+(V_i \ | \ 1 \leq i \leq d)$ in $D$.
Then $\ell_{D_{\p}}(A_{\p})=m+2d+2$.
\end{prop}

\begin{proof}
Let $\widetilde{k}=k[\{Z_i\}_{1 \leq i \leq d}, \{\frac{1}{Z_i}\}_{1 \leq i \leq d}]$ and $\widetilde{D}=D[\{\frac{1}{Z_i}\}_{1 \leq i \leq d}]$.
We set $X'_j=\frac{X_j}{Z_1}$ for $1 \leq j \leq m$, $V'_i=\frac{V_i}{Z_1}$ for $1 \leq i \leq d$, and $Y'=\frac{Y}{Z_1}$.
Then we have 
$$\widetilde{D}=\widetilde{k}[\{X'_j \ | \ 1 \leq j \leq m \}, Y', \{V'_i \ | \ 1 \leq i \leq d\}],$$
$$ \a \widetilde{D}=[(X'_j \ | \ 1 \leq j \leq m)+(Y')]\cdot[(X'_j \ | \ 1 \leq j \leq m)+(Y')+(V'_i \ | \ 1 \leq i \leq d)]$$
$$+(V'_iV'_j \ | \ 1 \leq i,j \leq d, \ i \neq j)+(\frac{Z^2_1}{Z_i}{V'_i}^3-Y' \ | \ 1 \leq i \leq d),$$
and $\{X'_j\}_{1 \leq j \leq m}$, $Y'$, and $\{V'_i\}_{1 \leq i \leq d}$ are algebraically independent over $\widetilde{k}$.
Let $$W=\widetilde{k}[\{X'_j \ | \ 1 \leq j \leq m\}, \{V'_i \ | \ 1 \leq i \leq d\}]$$ in $\widetilde{D}$ and 
$$ \b=[(X'_j \ | \ 1 \leq j \leq m)+({V'_1}^3)]\cdot[(X'_j \ | \ 1 \leq j \leq m)+(V'_i \ | \ 1 \leq i \leq d)] $$
$$ +(V'_i V'_j \ | \ 1 \leq i, j \leq d, \ i \neq j)+ (\frac{Z^2_1}{Z_i}{V'_i}^3-Z_1{V'_1}^3 \ | \ 2 \leq i \leq d)$$
in $W$.
Then substituting $Y'$ with $Z_1{V'_1}^3$ in $\widetilde{D}$, we get the isomorphism $$\widetilde{D}/\a \widetilde{U} \cong W/\b$$
of $\widetilde{k}$ algebras.
Then the prime ideal $\p \widetilde{D}/\a \widetilde{D}$ corresponds to the prime ideal $P/\b$ of $W/\b$, where $P=W_+=(X'_j \ | \ 1 \leq j \leq m)+(V'_i \ | \ 1 \leq i \leq d)$.
Then because
$$ \b+({V'_1}^3)=(X'_j \ | \ 1 \leq j \leq m)\cdot[(X'_j \ | \ 1 \leq j \leq m)+(V'_i \ | \ 1 \leq i \leq d)] $$
$$ + (V'_iV'_j \ | \ 1 \leq i,j \leq d, \ i \neq j)+({V'_i}^3 \ | \ 1 \leq i \leq d) $$
and $\ell_{W_P}([\b+({V'_1}^3)]W_P/\b W_P)=1$, we get
\begin{eqnarray*}
\ell_{W_P/\b W_P}(W_P/\b W_P)&=&\ell_{W_P}(W_P/[\b+({V'_1}^3)]W_P)+\ell_{W_P}([\b+({V'_1}^3)]W_P/\b W_P)\\
&=& (m+2d+1)+1=m+2d+2.
\end{eqnarray*}
Thus $\ell_{A_{\p}}(A_{\p})=\ell_{W_P/\b W_P}(W_P/\b W_P)=m+2d+2$.
\end{proof}

Thanks to the associative formula of multiplicity, we have
$$ \e_0(Q)=\ell_{A_{\p}}(A_{\p})\cdot\e_0^{A_{\p}/\p A_{\p}}([Q+\p A]/\p A)=m+2d+2, $$
because $\p=\sqrt{\a}$ and $A/\p A=D/\p \cong k[[Z_i \ | \ 1 \leq i \leq d]]$.
On the other hand, we have
$$A/Q \cong k[[\{X_j\}_{1 \leq j \leq m}, Y, \{V_i\}_{1 \leq i \leq d}]]/\mathfrak{c}$$
where 
$$\mathfrak{c}=([(X_j \ | \ 1 \leq j \leq m)+(Y)]\cdot[(X_j \ | \ 1 \leq j \leq m)+(Y)+(V_i \ | \ 1 \leq i \leq d)] $$
$$ + (V_iV_j \ | \ 1 \leq i,j \leq d, \ i \neq j)+({V_i}^3 \ | \ 1 \leq i \leq d).$$
Therefore $\ell_A(A/Q)=m+2d+2$.
Thus $\e_0(Q)=\ell_A(A/Q)$ so that $A$ is a Cohen-Macaulay local ring with $\e_0(Q)=m+2d+2$.

Let $K_0=A$, $K_1=\m$, and $K_n=\m^n+y\m^{n-2}$ for $n \geq 2$, and we set $\mathcal{K}=\{K_n\}_{n \geq 0}$.
Let $\e_i(\calK)$ denote the $i$th Hilbert coefficients of the filtration $\calK$ for $0 \leq i \leq d$.

\vskip 2mm

\vskip 2mm
\begin{lem}\label{ex4}
The following assertions hold true.
\begin{itemize}
\item[(1)] $\ell_A(K_2/\m^2)=1$ and $K_n=\m^n$ for all $n \geq 3$.
\item[(2)] $Q \cap K_2=QK_1$ and $K_{n+1}=QK_n$ for all $n \geq 2$. Therefore $\e_1(\mathcal K)=\e_0(\mathcal K)-\ell_A(A/K_1)+\ell_A(K_2/Q K_1)$, $\e_2(\calK)=\ell_A(K_2/QK_1)$ if $d \geq 2$, and $\e_i(\calK)=0$ for $3 \leq i \leq d$.
\end{itemize}
\end{lem}

\begin{proof}
(1) Since $K_2=\m^2+(y)$, we have $\ell_A(K_2/\m^2)=1$.
We have $K_n=\m^n+y\m^{n-2}=\m^n$ for all $n \geq 3$, because $y\m=(yv_i \ | \ 1 \leq i \leq d)= ({v_i}^3 \ | \ 1 \leq i \leq d) \subseteq \m^3$.

$(2)$ Since $K_n=\m^n$ for all $n \geq 3$ by assertion $(1)$, we have $K_2 \subseteq \widetilde{\m^2}$. 
Therefore $Q \cap K_2 \subseteq Q \cap \widetilde{\m^2}=Q\m=QK_1$ by Remark \ref{tilde}.
It is routine to check that $K_{n+1}=QK_n$ for all $n \geq 2$.
Thus $\e_1(\mathcal K)=\e_0(\mathcal K)-\ell_A(A/K_1)+\ell_A(K_2/Q K_1)$ by \cite[Theorem 2.2]{GR}.
We also have $\e_2(\calK)=\ell_A(K_2/QK_1)$ if $d \geq 2$, and $\e_i(\calK)=0$ for $3 \leq i \leq d$ by \cite[Proposition 4.6]{HM}.
\end{proof}

We prove now  Theorem \ref{ex2}.

\begin{proof}[Proof of Theorem \ref{ex2}]
Since $K_n=\m^n$ for all $n \geq 3$ by Lemma \ref{ex4} (1), we have $\e_i(\mathcal K)=\e_i(\m)$ for $0 \leq i \leq d$.
Therefore, $\e_1(\m)=\e_0(\m)+\ell_A(\m^2/Q\m)$, $\e_2(\m)=\ell_A(\m^2/Q\m)+1$ if $d \geq 2$, and $\e_i(\m)=0$ for all $3 \leq i \leq d$, because $\ell_A(K_2/\m^2)=1$, $\e_1(\mathcal K) = \e_0(\mathcal K)+\ell_A(K_2/QK_1)-1$, $\e_2(\calK)=\ell_A(K_2/QK_1)$ if $d \geq 2$, and $\e_i(\calK)=0$ for $3 \leq i \leq d$ by Lemma \ref{ex4}.
Then we have $\e_1(\m)=m+3d+2$ and $\e_2(\m)=d+1$ because $\e_0(\m)=m+2d+2$ and $\ell_A(\m^2/Q\m)=d$.
The ring $\rmG(\m)$ is Buchsbaum ring with $\depth \rmG(\m)=0$ and $\Bbb{I}(\rmG(\m))=d$ by Theorem \ref{B_+}.
This completes the proof of Theorem \ref{ex2}.
\end{proof}




\begin{thebibliography}{GNO}



\bibitem[EV]{EV} J. Elias and G. Valla, \textit{Rigid Hilbert functions}, J. Pure and Appl. Algebra $\bf71$ (1991) 19--41.





\bibitem[GNO]{GNO} S. Goto, K. Nishida, and K. Ozeki, \textit{Sally modules of rank one}, Michigan Math. J. $\bf57$ (2008) 359--381.





\bibitem[GR]{GR} A. Guerrieri and M. E. Rossi, {\it Hilbert coefficients of Hilbert filtrations}, J. Algebra ${\bf 199}$ (1998) 40--61.

\bibitem[H]{H} C. Huneke, \textit{Hilbert functions and symbolic powers}, Michigan Math. J. ${\bf 34}$ (1987) 293--318.



\bibitem[HM]{HM} S. Huckaba and T. Marley, \textit{Hilbert coefficients and the depth of associated graded rings}, J. London Math. Soc. (2) $\bf56$ (1997) 64--76.

\bibitem[I1]{I1} S. Itoh, \textit{Hilbert coefficients of integrally closed ideals}, J. Algebra ${\bf 176}$ (1995) 638--652.


\bibitem[I2]{I2} S. Itoh, \textit{Integral closures of ideals generated by regular sequence}, J. Algebra ${\bf 117}$ (1988) 390--401.



\bibitem[N]{N} D. G. Northcott, \textit{A note on the coefficients of the abstract Hilbert function}, J. London Math. Soc. ${\bf 35}$ (1960) 209--214.






\bibitem[RR]{RR} L. J. Ratliff and D. Rush, \textit{Two notes on reductions of ideals}, Indiana Univ. Math. J. ${\bf 27}$ (1978) 929--934.



\bibitem[RV1]{RV1} M. E. Rossi and G. Valla \textit{A conjecture of J. Sally}, Comm. Alg. ${\bf 24}$ (13) (1996) 4249--4261.

\bibitem[RV2]{RV2} M. E. Rossi and G. Valla \textit{Cohen-Macaulay local rings of embedding dimension $e+d-3$}, Proc. London Math. Soc. (3) ${\bf 80}$ (2000) 107--126.


\bibitem[RV3]{RV3} M. E. Rossi and G. Valla \textit{Hilbert functions of filtered modules}, UMI Lecture Notes 9, Springer (2010).





\bibitem[S1]{S1} J. D. Sally, \textit{Hilbert coefficients and reduction number 2}, J. Alg. Geo. and Sing. ${\bf 1}$ (1992) 325--333.

\bibitem[S2]{S2} J. D. Sally, \textit{Ideals whose Hilbert function and Hilbert polynomial agree at $n=1$}, J. Algebra $\bf157$ (1993), 534--547.


\bibitem[V]{V} W. V. Vasconcelos, \textit{Hilbert Functions, Analytic Spread, and Koszul Homology}, Contemporary Mathematics, Vol $\bf159$ (1994) 410--422.

\bibitem[VP]{VP} M. Vaz Pinto, \textit{Hilbert functions and Sally modules}, J. Algebra, $\bf192$ (1996) 504--523.

\bibitem[W]{W} H.-J. Wang, {\it Links of symbolic powers of prime ideals,} Math. Z. ${\bf 256}$ (2007) 749--756.


\end{thebibliography}
\end{document}